\documentclass[12pt]{amsart}

\usepackage{amssymb}
\usepackage{color}
\usepackage{graphicx}
\usepackage{mathrsfs}

\usepackage[citecolor=blue]{hyperref}

\numberwithin{equation}{section}
\newtheorem{theorem}{Theorem}[section]

\newtheorem{lemma}[theorem]{Lemma}

\DeclareMathOperator{\diver}{div}
\DeclareMathOperator{\trace}{Trace}

\newcommand{\supp}{\text{\rm supp}}

\title[The NS equations in the plane]{Inverse boundary value problem for the
Stokes and the Navier-Stokes equations in the plane}

\author[Lai]{Ru-Yu Lai}
\address{Department of Mathematics, University of Washington, Seattle, WA 98195-4350,
USA}
\curraddr{}
\email{rylai@math.washington.edu }
\thanks{}
\author[Uhlmann]{Gunther Uhlmann}
\address{Department of Mathematics, University of Washington, Seattle, WA 98195-4350,
USA}
\curraddr{}
\email{gunther@math.washington.edu }
\thanks{The first and second author were supported in part by the National Science Foundation. The second author was also supported by a Simons Fellowship.}
\author[Wang]{Jenn-Nan Wang}
\address{Institute of Applied Mathematics, NCTS (Tapei), National Taiwan University, Taipei 106, Taiwan}
\curraddr{}
\email{jnwang@math.ntu.edu.tw}
\thanks{The third author was supported in part by the National Science Council of Taiwan. }

\date{}

\begin{document}

%\tableofcontents

\begin{abstract}
   In this paper, we prove in two dimensions global identifiability of the viscosity in an incompressible fluid by making boundary measurements.
%  This problem was previously considered by Imanuvilov and Yamamoto \cite{IY}
%in which they showed that the Dirichlet-to-Neumann map (with the information on the pressure) uniquely determines the viscosity.
The main contribution of this work is to use more natural boundary measurements, the Cauchy forces, than the Dirichlet-to-Neumann map previously considered in \cite{IY} to prove the uniqueness of the viscosity for the Stokes equations and for the Navier-Stokes equations.
\end{abstract}
\maketitle

\tableofcontents

\section{Introduction}

Let $\Omega$ be a simply connected bounded domain in $\mathbb{R}^2$
with smooth boundary.
Assume that $\Omega$ is filled with an incompressible fluid.
Let $u=(u_1,u_2)^T$ be the velocity vector field satisfying the Stokes equations
\begin{align}\label{stokes}
       \left\{
          \begin{array}{cl}
         \diver\sigma(u,p)=0 & \hbox{in $\Omega$},\\
            \diver u=0  & \hbox{in $ \Omega$},
           \end{array}
        \right.
\end{align}
where $\sigma(u,p)=2\mu\varepsilon-pI_2$ is the stress tensor and $\varepsilon=((\nabla u)+(\nabla u)^T)/2$, $\mu$ is the viscosity and $p$ is the pressure. Here the notation $I_2$ is the $2\times 2$ identity matrix.

Physically,
zero viscosity is observed only in superfluids the have the ability to self-propel and travel in a way that
defies the forces of gravity and surface tension. Otherwise all fluids have positive viscosities.
Thus, we can assume that $\mu>0$ in $\overline\Omega$. The second equation of (\ref{stokes}) is the incompressibility condition. Because of the conservation of mass, the incompressibility condition is equivalent to the material derivative of the density function $\rho$ to be zero, that is,
\begin{align}\label{rho}
     \frac{D\rho}{Dt}=:\frac{\partial\rho}{\partial t}+u\cdot \nabla \rho=0.
\end{align}
When $\rho$ is constant, (\ref{rho}) is satisfied. The above equation also holds for nonconstant density functions. We can conclude that a nonconstant viscosity $\mu$ is possible since the viscosity function is a function of density. A fluid with nonconstant viscosity is called a non-Newtonian fluid which is relatively common, such as blood, shampoo and custard.

Let $g\in H^{3/2}(\partial\Omega)$ satisfy the compatibility condition
\begin{align}\label{compat}
\int_{\partial\Omega} g\cdot \textbf{n} dS=0,
\end{align}
where $\textbf{n}$ is the unit outer normal to $\partial\Omega$. This condition leads to the uniqueness of (\ref{stokes}), that is,
there exists a unique solution $(u,p)\in H^2(\Omega)\times H^1(\Omega)$
($p$ is unique up to a constant) of (\ref{stokes}) and $u|_{\partial\Omega}=g$.
We could define the Cauchy data for the Stokes equations (\ref{stokes}) by
$$
     C_{\mu}=\left\{ \left(u, \sigma(u,p)\textbf{n}\right)|_{\partial\Omega}:
      (u,p)\ \mbox{satisfies (\ref{stokes})}  \right\}.
$$
%In this paper, we consider the Cauchy data $C_\mu$ contains the value $\sigma(u,p)\textbf{n}$.
The inverse problems we consider in this paper is to determine $\mu$ from the knowledge of the Cauchy data $C_\mu$.
Recently, Imanuvilov and Yamamoto \cite{IY}  studied the same inverse problem with the Dirichlet-to-Neumann (DN) map defined by
$$
\Lambda_\mu(g)=\left(\frac{\partial u}{\partial\nu}, p\right)\Big|_{\partial\Omega}.
$$
They showed that the knowledge of the DN map uniquely determines the viscosity $\mu$ of the Navier-Stokes equations. Unlike their boundary measurements, we use Cauchy data $\left(u, \sigma(u,p)\textbf{n}\right)|_{\partial\Omega}$ to deduce the uniqueness of $\mu$. The physical sense of $\sigma(u,p)\textbf{n}$ is the stress
acting on $\partial\Omega$ and is called the Cauchy force. Mathematically, the pressure function $p$ plays the role of the Lagrange multiplier corresponding to the incompressibility condition. The information of $p$ on $\partial\Omega$ is coupled in the Cauchy force. Given the measurement of $p$ alone on $\partial\Omega$ is unnatural.

%\HOX{regularity of $p$}
The main result of this paper is the following global uniqueness result. Note that the following theorem also holds for the Navier-Stokes equations.

\begin{theorem}\label{main}
     Let $\Omega$ be a simply connected bounded domain in $\mathbb{R}^2$ with smooth boundary. Suppose that $\mu_1$ and $\mu_2$ are two viscosity functions for the Stokes equations. Assume that $\mu_j\in C^3(\overline\Omega)$ and $\mu_j>0$ with
     \begin{align*}
     \partial^\alpha\mu_1|_{\partial\Omega}=\partial^\alpha\mu_2|_{\partial\Omega}\ \ \hbox{for all $|\alpha|\leq 1$.}
     \end{align*}
     Let $C_{\mu_1}$ and $C_{\mu_2}$ be the Cauchy data associated with $\mu_1$ and $\mu_2$, respectively. If $ C_{\mu_1}=C_{\mu_2}$, then
$
\mu_1=\mu_2
$
in $\Omega$.
\end{theorem}
In higher dimensions, the global uniqueness of identifying the viscosity using the Cauchy data has been well studied.
For the Stokes equations, the uniqueness for the inverse boundary problem was established by Heck, Li and Wang \cite{HLW} in dimension three.
In \cite{LW}, Li and Wang proved the unique determination of $\mu$ for the Navier-Stokes equations in dimension three. To study the Navier-Stokes equations they applied the linearization technique due to Isakov \cite{is93}. The idea is to reduce the semilinear inverse boundary value problem to the corresponding linear one. When applying the linearization method to the Navier-Stokes equation, the difficulty is to show the existence of particular solutions to the Navier-Stokes equations with certain controlled asymptotic properties. The idea used in \cite{LW} is independent of the spatial dimension. It works for the two-dimensional case as well. We will briefly describe the result in Section \ref{sec4}.
%In dimension two, in 2013,
%Imanuvilov and Yamamoto \cite{IY} showed that the Dirichlet-to-Neumann map
%$$
%\Lambda_\mu(g)=\left(\frac{\partial u}{\partial\nu}, p\right)\Big|_{\partial\Omega}.
%$$
%uniquely determines the viscosity $\mu$ in the $2$D Navier-Stokes equations.

Our first strategy for proving Theorem \ref{main} is to show that the inverse boundary value problem for the 2D Stokes equations and that for the thin plate-like are equivalent. The equivalence is known to hold for the 2D isotropic elastic equation and the thin plate equation. Recently, Kang, Milton and Wang \cite{KMW} gave explicit formulas showing that the Cauchy data of the elasticity system determines the Cauchy data of the thin plate equations, and vice versa (see also \cite{ik94}). Since the Stokes equations can be viewed as an elasticity system with incompressibility, we can prove a similar equivalence by using the similar arguments in \cite{KMW}.
Having established the equivalence of two inverse boundary value problems, we then transform the thin plate equations into a first order system. Albin, Guillarmou, Tzou and Uhlmann \cite{AGTU} showed that the Cauchy data of the first order system $D+V$ uniquely determine $V$ if $V$ is diagonal, where $D$ is an operator with $\partial$ or $\bar\partial$ at its diagonal. When $V$ is not diagonal, they reduced it to the diagonal case so that the similar result holds for the non-diagonal one. For the Stokes equation, the potential $V$ contains the function $\mu$ up to the second order derivative, we apply their result and the assumption on the boundary of $\mu$ to deduce the global uniqueness.

The paper is organized as follows. We show the equivalence of the inverse boundary problems for the thin plate-like
and for the Stokes equations in Section \ref{sec2}. In Section \ref{sec3}, we derive a first order system from the thin plate-like equations. We then show that the Cauchy data of the first order system uniquely determines the viscosity $\mu$. In Section \ref{sec4}, we study the same inverse problem for the Navier-Stokes equations.
%\textcolor{red}{Finally, a useful lemma is proved in Appendix \ref{app:p}.}

%\section{Preliminaries}

\section{Equivalence of boundary data for the plate and for the 2D Stokes equations}\label{sec2}
In this section we would like to connect the inverse boundary value problem for the thin plate equations to that for the Stokes equations. We define the 4-th order tensor $\mathcal{R}$ by
$$
\mathcal{R}M=R^T_\perp MR_\perp
$$
for any $2\times 2$ matrix $M$, where
$$
     R_\perp=\left(
               \begin{array}{cc}
                 0 & 1 \\
                 -1 & 0 \\
               \end{array}
             \right).
$$
Hereafter, for any function $u$, the notation $u_{,j}$ means the
derivative of $u$ with respect to $x_j$, $j=1,2$. Denote $\sigma=(\sigma_{ij})$.Componentwise, the first equation of (\ref{stokes}) is equivalent to
$$
       \sigma_{11,1}+\sigma_{12,2}=0,\ \ \sigma_{21,1}+\sigma_{22,2}=0.
$$
It follows that there exist potentials $\psi_1$ and $\psi_2$ such that
\begin{align}\label{sigma}
     \sigma_{11}=\psi_{1,2},\ \ \sigma_{12}=-\psi_{1,1},\ \ \sigma_{21}=\psi_{2,2},\ \ \sigma_{22}=-\psi_{2,1}.
\end{align}
Since $\sigma_{12}=\sigma_{21}$, we have
$$
      \psi_{1,1}+\psi_{2,2}=0.
$$
Thus there exists a potential $\phi$ such that
\begin{align}\label{airy}
     \psi_1=\phi_{,2},\ \ \psi_2=-\phi_{,1}.
\end{align}
The potential $\phi$ is called the Airy stress function. Substituting (\ref{airy}) into (\ref{sigma}), we see that
\begin{align}\label{vare}
    \sigma=\mathcal{R}\nabla^2\phi
    = \left(
                    \begin{array}{cc}
                      \phi_{,22} & -\phi_{,12} \\
                      -\phi_{,21} & \phi_{,11} \\
                    \end{array}
      \right),
\end{align}
where $\nabla^2\phi$ denotes the Hessian of $\phi$, i.e.,
$$
     \nabla^2\phi=\left(
                    \begin{array}{cc}
                      \phi_{,11} & \phi_{,12} \\
                      \phi_{,21} & \phi_{,22} \\
                    \end{array}
                  \right).
$$

In light of $\sigma=2\mu\varepsilon-pI_2$ and (\ref{vare}), we get
\begin{align}\label{vare1}
     \varepsilon=\frac{1}{2\mu} \left(\mathcal{R}\nabla^2\phi+pI_2\right).
\end{align}
The divergence-free condition $\diver u=0$ implies that
$$
     0=u_{1,1}+u_{2,2}=\trace(\varepsilon)=\frac{1}{2\mu} \left(\Delta\phi+2p\right),
$$
thus
\begin{align}\label{p}
     p=-\frac{\Delta\phi}{2}.
\end{align}
Note that the physical significance of the pressure $p$ is that $-p$ is the mean of the two normal stresses
at a point, that is,
$$
   p=-\frac{1}{2}(\sigma_{11}+\sigma_{22}).
$$

From (\ref{vare1}) and (\ref{p}), it follows that
\begin{align}\label{divdiv}
0=\diver\diver(\mathcal{R}\varepsilon)
&=\diver\diver\left(\frac{1}{2\mu}
\mathcal{R}\left(\mathcal{R}\nabla^2\phi+pI_2\right)\right)\notag \\
&=\diver\diver\left(\frac{1}{2\mu}
\left(\nabla^2\phi-\frac{\Delta\phi}{2} I_2\right)\right)\notag \\
&=\diver\diver\left(\frac{1}{2\mu}
\left(
  \begin{array}{cc}
    \frac{1}{2}(\phi_{,11}-\phi_{,22}) & \phi_{,12} \\
    \phi_{,12} & \frac{1}{2}(\phi_{,22}-\phi_{,11})  \\
  \end{array}
\right)
\right).
\end{align}
Conversely, if $\diver\diver(\mathcal{R}\varepsilon)=0$, then
$$
     \varepsilon_{22,11}+\varepsilon_{11,22}-2\varepsilon_{12,12}=0,
$$
where $\varepsilon=(\varepsilon_{ij})$.
If $\Omega$ is simply connected, then there exists a function $u$ such that $\varepsilon=\left((\nabla u)+(\nabla u)^T\right)/2.$
(For the proof of the existence of such function $u$, we refer to \cite{F}, page 99-103). Based on (\ref{divdiv}), the function $u$ also satisfies $\diver u=0$. Let $p=-\Delta \phi/2$, then $(u,p)$ satisfies
the Stokes equations (\ref{stokes}).
Thus we have proved that the two systems (\ref{stokes}) and (\ref{divdiv}) are equivalent if $\Omega$ is simply
connected.

Next we would like to discuss the equivalence of the Cauchy data. We define the operator $P_\mu(\phi)$ by
\begin{align*}
P_\mu(\phi):=\diver\diver\left(\frac{1}{2\mu}
\left(
  \begin{array}{cc}
    \frac{1}{2}(\phi_{,11}-\phi_{,22}) & \phi_{,12} \\
    \phi_{,12} & \frac{1}{2}(\phi_{,22}-\phi_{,11})  \\
  \end{array}
\right)
\right)
\end{align*}
and denote $u_{,n}=\nabla u\cdot
\textbf{n}$ and $u_{,t}=\nabla u\cdot \textbf{t}$, where $\textbf{n}=(n_1,n_2)$ is the unit normal and
$\textbf{t}=(-n_2,n_1)=R_{\perp}^T\textbf{n}$ is the unit tangent vector field along
$\partial\Omega$ in the positive orientation. The Dirichlet data associated with (\ref{divdiv}) is described by
the pair $\{\phi,\phi_{,n}\}$ and the Neumann data by the pair
\begin{align*}
     M_n&:=\textbf{n}\cdot\left(\frac{1}{2\mu}
\left(\nabla^2\phi-\frac{\Delta\phi}{2} I_2\right)\right)\textbf{n},\\
     (M_t)_{,t}&:= \diver \left(\frac{1}{2\mu}
\left(\nabla^2\phi-\frac{\Delta\phi}{2} I_2\right)\right)\cdot\textbf{n}
+\left(\textbf{t}\cdot\frac{1}{2\mu}
\left(\nabla^2\phi-\frac{\Delta\phi}{2} I_2\right)\textbf{n}\right)_{,t}.
\end{align*}
We define the Cauchy data for (\ref{divdiv}) by
$$
   C^*_{\mu}=\left\{ \left(\phi,\phi_{,n},M_n,(M_t)_{,t}\right)|_{\partial\Omega}:
          \phi\in H^4(\Omega), P_\mu(\phi)=0 \right\}.
$$

We now adopt the arguments used in \cite{KMW} to show that $\sigma \textbf{n}|_{\partial\Omega}$ determines
$\{\phi,\phi_{,n}\}$ on $\partial\Omega$ and $u|_{\partial\Omega}$ determines $\{M_n,(M_t)_{,t}\}$ on $\partial\Omega$, and vice versa.
Therefore, the Cauchy data $C_\mu$ for the Stokes equations and the Cauchy data $C^*_\mu$ for (\ref{divdiv}) are equivalent. Assume for the moment that $u\in C^{2+\alpha}(\overline\Omega)$ for some $\alpha\in(0,1)$.
It follows from (\ref{vare}) that
\begin{align}\label{rt}
    R^T_\perp\sigma\textbf{n}= (\nabla^2\phi) \textbf{t}=\left(
  \begin{array}{cc}
    \nabla \phi_{,1}\cdot\textbf{t} \\
    \nabla \phi_{,2}\cdot\textbf{t} \\
  \end{array}
\right).
\end{align}
For $j=1,2$, we integrate $ \nabla \phi_{,j}\cdot\textbf{t}$ along $\partial\Omega$ from some point $x_0\in\partial\Omega$,
we recover $\nabla\phi$ (up to a constant) on $\partial\Omega$. Hence $\phi_{,n}$ and $\phi_{,t}$ are recovered. We integrate $\phi_{,t}$ along $\partial\Omega$, then $\phi$ on $\partial\Omega$ is known (also up to a constant). The appearance of integrating constants is evident from \eqref{divdiv}.  In other words, the traction $\sigma\textbf{n}$ uniquely determines the Dirichlet data $\phi$ and $\phi_{,n}$. On the other hand, if $\phi$ and $\phi_{,n}$ are given, then $\nabla\phi$ is known. Hence, the boundary traction $\sigma\textbf{n}$ is recovered via (\ref{rt}).

To show that $M_n$ and $(M_t)_{,t}$ can be recovered from $u$.
Since $\varepsilon=((\nabla u)+(\nabla u)^T)/2$, we get that
\begin{align*}
    \mathcal{R}\varepsilon=R^T_\perp\varepsilon R_\perp=
    \left(
      \begin{array}{cccc}
        u_{2,2} & -\frac{1}{2}(u_{2,1}+u_{1,2}) \\
        -\frac{1}{2}(u_{2,1}+u_{1,2}) & u_{1,1}  \\
      \end{array}
    \right)
\end{align*}
and thus
\begin{align*}
    \diver{(\mathcal{R}\varepsilon)}&=
    \left(
      \begin{array}{cccc}
        u_{2,21}  -\frac{1}{2}(u_{2,12}+u_{1,22}) \\
        -\frac{1}{2}(u_{2,11}+u_{1,21}) +u_{1,12}  \\
      \end{array}
    \right)\\
    &=\frac{1}{2}\left(
      \begin{array}{cccc}
        u_{2,12} -u_{1,22} \\
        u_{1,12}-u_{2,11} \\
      \end{array}
    \right)
    =\frac{1}{2}R^T_{\perp}
    \nabla
    \left(
        u_{1,2} -u_{2,1}
    \right).
\end{align*}
Consequently, we obtain
\begin{align*}
   \diver{(\mathcal{R}\varepsilon)}\cdot \textbf{n} &=
   \frac{1}{2}(R_{\perp}\textbf{n}) \cdot\nabla
    \left(
        u_{1,2} -u_{2,1}
    \right)=
   -\frac{1}{2}\textbf{t}\cdot
    \nabla
    \left(
        u_{1,2} -u_{2,1}
    \right)\notag\\
    &= \frac{1}{2}\left(
        u_{2,1} -u_{1,2}
    \right)_{,t}.
\end{align*}
%Since $\diver\diver(\mathcal{R}\varepsilon)=0$, there is a potential $\psi$ such that
%$$
%     \diver(\mathcal{R}\varepsilon)=(\psi_{,2},-\psi_{,1})
%$$
%and hence
%\begin{align}\label{psi}
%     \diver(\mathcal{R}\varepsilon)\cdot \textbf{n}=\nabla\psi\cdot\textbf{t}=\psi_{,t}\ \ \hbox{on $\partial\Omega$}.
%\end{align}
%From (\ref{n}) and (\ref{psi}), we obtain that
%\begin{align}\label{psi1}
%\psi_{,t}=\frac{1}{2}\left(
%        u_{2,1} -u_{1,2}
%    \right)_{,t} .
%\end{align}
Recall that $(M_t)_{,t}=\diver(\mathcal{R}\varepsilon)\cdot \textbf{n}+(\textbf{t}\cdot (\mathcal{R}\varepsilon)\textbf{n})_{,t}$ and therefore
\begin{align}\label{Mtt}
     (M_t)_{,t}=\frac{1}{2}\left(
        u_{2,1} -u_{1,2}
    \right)_{,t}+(\textbf{t}\cdot (\mathcal{R}\varepsilon)\textbf{n})_{,t}.
\end{align}
Integrating (\ref{Mtt}) along $\partial\Omega$ from some point $x_0\in\partial\Omega$ and choosing an appropriate $(u_{2,1} -u_{1,2})(x_0)$, we obtain
\begin{align}\label{mt:R}
     M_t=\frac{1}{2}\left(
        u_{2,1} -u_{1,2}    \right)+ \textbf{t}\cdot (\mathcal{R}\varepsilon)\textbf{n}.
\end{align}
We observe that
\begin{align*}
     \frac{1}{2}\left(
        u_{2,1} -u_{1,2}    \right)
        =\left(R^T_\perp
        \left(
          \begin{array}{cc}
            0 & \frac{1}{2}(u_{1,2}-u_{2,1}) \\
            \frac{1}{2}(u_{2,1}-u_{1,2}) & 0 \\
          \end{array}
        \right)
        R_\perp\textbf{n}\right)\cdot\textbf{t}.
\end{align*}
The second term on the right side of (\ref{mt:R}) can be written as
$$
   \textbf{t}\cdot (\mathcal{R}\varepsilon)\textbf{n}=
   \left(R^T_\perp
        \left(
          \begin{array}{cc}
            u_{1,1} & \frac{1}{2}(u_{1,2}+u_{2,1}) \\
            \frac{1}{2}(u_{1,2}+u_{2,1}) & u_{2,2} \\
          \end{array}
        \right)
        R_\perp\textbf{n}\right)\cdot\textbf{t}.
$$
Thus we have
\begin{align}\label{mt}
  M_t&=\left( R^T_\perp
        \left(
          \begin{array}{cc}
            u_{1,1} & u_{1,2} \\
           u_{2,1} & u_{2,2} \\
          \end{array}
        \right)
        R_\perp\textbf{n}\right)\cdot\textbf{t}\notag\\
        &= -R^T_\perp( \nabla u)\textbf{t}\cdot\textbf{t}\notag\\
        &=-\textbf{n}\cdot( \nabla u)\textbf{t}.
\end{align}
Moreover, using the definition of $M_n$, we get
\begin{align}\label{mn}
   M_n=\textbf{n}\cdot \mathcal{R}\varepsilon\textbf{n}=\textbf{n}\cdot R^T_\perp \varepsilon R_\perp \textbf{n}=\textbf{t}\cdot( \nabla u)\textbf{t}.
\end{align}
From (\ref{mt}) and (\ref{mn}), we deduce that
\begin{align}\label{mnt}
     u_{,t}=-M_t\textbf{n}+M_n\textbf{t},
\end{align}
which implies the Neumann data $M_n$ and $M_t$ can be recovered from $u_{,t}$.
On the other hand, we use the formula (\ref{mnt}) and integrate $-M_t\textbf{n}+M_n\textbf{t}$ along $\partial\Omega$. Thus, the velocity field $u$ is determined.

By a density argument the above discussion holds for the slightly relaxed regularity assumption on the boundary data $g\in H^{3/2}(\partial\Omega)$. Hence, we can remove the assumption that $u\in C^{2+\alpha}(\overline\Omega)$. We therefore conclude that knowing the Cauchy data of the Stokes equations is equivalent to knowing that of the thin plate-like equations \eqref{divdiv}.

\section{Global uniqueness for the Stokes equations}\label{sec3}

From the previous section, we have concluded that to study the inverse boundary value for the Stokes equations \eqref{stokes}, it suffices to consider the same question for the plate-like equation \eqref{divdiv}. Our strategy now is to deduce a first order system $DU+VU=0$ from \eqref{divdiv}. The most nontrivial property that we will show is that $C_{\mu}^{\ast}$ determines the Cauchy data of the first order system $DU+VU=0$. Having obtained this result, the global identifiability of $\mu$ for the Stokes equations is reduced to the uniqueness problem for this first order system. The global uniqueness of the inverse boundary value problem for such a first order system was recently studied by Albin, Guillarmou, Tzou and Uhlmann in \cite{AGTU}. Consequently, the proof of the uniqueness question for the Stokes equations follows from their result.

\subsection{$(\partial^2_{\overline z},\partial^2_z)$ system}\label{3.1}
%\section{ }\label{app:p}
As usual, we define $z=x+iy$,
$$
     \partial_z=\frac{1}{2}\left(\frac{\partial}{\partial x}-i\frac{\partial}{\partial y}\right),
\ \ \ \partial_{\overline z}=\frac{1}{2}\left(\frac{\partial}{\partial x}+i\frac{\partial}{\partial y}\right).
$$
The complex version of Gauss integral formulas are given by
\begin{align}\label{Gauss}
     \int_\Omega \partial_{\overline z} w(z) dxdy=\frac{1}{2i}\int_{\partial\Omega} w(z)dz, \ \
     \int_\Omega \partial_z w(z) dxdy=-\frac{1}{2i}\int_{\partial\Omega} w(z)d\overline z
\end{align}
for $w\in C^1(\Omega)\cap C(\overline\Omega)$ lead to the Cauchy
Pompeiu representations
\begin{align}
   w(z)=\frac{1}{2\pi i}\int_{\partial\Omega} w(\zeta)\frac{d\zeta}{\zeta-z}-\frac{1}{\pi}\int_{\Omega} \partial_{\overline\zeta} w(\zeta)\frac{d\xi d\eta}{\zeta-z},\ \ z\in\Omega,\label{Pomp1}\\
   w(z)=-\frac{1}{2\pi i}\int_{\partial\Omega} w(\zeta)\frac{d\overline\zeta}{\overline{\zeta-z}}-\frac{1}{\pi}\int_{\Omega} \partial_{\zeta} w(\zeta)\frac{d\xi d\eta}{\overline{\zeta-z}},\ \ z\in\Omega,\label{Pomp2}
\end{align}
where $\zeta=\xi+i\eta$.  Iterations of these formulas give the following
higher order representations
\begin{align}\label{hpomp}
   w(z)&=\frac{1}{2\pi i}\int_{\partial\Omega} w(\zeta)\frac{d\zeta}{\zeta-z}-\frac{1}{2\pi i}\int_{\partial\Omega} \partial_{\overline\zeta} w(\zeta)\frac{\overline{\zeta-z}}{\zeta-z}d\zeta
+\frac{1}{\pi}\int_{\Omega} \partial^2_{\overline\zeta} w(\zeta)\frac{\overline{\zeta-z}}{\zeta-z}d\xi d\eta
\end{align}
and
\begin{align}\label{hpomp2}
   w(z)&=-\frac{1}{2\pi i}\int_{\partial\Omega} w(\zeta)\frac{d\overline\zeta}{\overline{\zeta-z}}+\frac{1}{2\pi i}\int_{\partial\Omega} \partial_{\zeta} w(\zeta)\frac{\zeta-z}{\overline{\zeta-z}}d\overline\zeta
+\frac{1}{\pi}\int_{\Omega} \partial^2_{\zeta} w(\zeta)\frac{\zeta-z}{\overline{\zeta-z}}d\xi d\eta.
\end{align}
for $w\in C^2(\Omega)\cap C^1(\overline\Omega)$ (see \cite[Page 272]{B}). In the sequel, we need a technical lemma.
\begin{lemma}\label{lem3.1}
     Let $\Omega$ be an open bounded domain in $\mathbb{C}$ and $f\in C^k(\overline\Omega)$ for $k\geq 2$. Define
     $$
          u(z)=\frac{1}{\pi}\int_{\Omega} f(\zeta)\frac{\overline{z-\zeta}}{z-\zeta}d\xi d\eta
     $$
     Then $u(z)$ is in $C^k(\Omega)$ and
     satisfies
     \begin{equation}\label{1422}
        \partial_{\overline z}^2 u(z)=f(z)
     \end{equation}
     in $\Omega$. Likewise, if we define
      $$
          u(z)=\frac{1}{\pi}\int_{\Omega} f(\zeta)\frac{z-\zeta}{\overline{z-\zeta}}d\xi d\eta,
     $$
     then $u(z)$ is in $C^k(\Omega)$ and
     satisfies
     $$
        \partial_{z}^2 u(z)=f(z)
     $$
      in $\Omega$.
\end{lemma}
\begin{proof}
We adopt the proof of a similar result in \cite[Theorem~2.1.2]{CS} to our case here. We only prove the first part of the lemma, the other part is treated similarly. We first consider $f\in C^k_0(\mathbb{C})$. Changing variable $\zeta'=z-\zeta$ in $u$ and differentiation under the integral sign implies that $u\in C^k(\mathbb{C})$. To verify \eqref{1422}, we apply Gauss integral formula twice and (\ref{hpomp}) (note that $f$ is compactly supported). We get
\begin{align*}
\partial^2_{\overline z} u(z)=\frac{1}{\pi}\int_{\mathbb{C}}  f(\zeta)\partial^2_{\overline \zeta}\frac{\overline{z-\zeta}}{z-\zeta}d\xi d\eta
=\frac{1}{\pi}\int_{\mathbb{C}}  \partial^2_{\overline \zeta}f(\zeta)\frac{\overline{z-\zeta}}{z-\zeta}d\xi d\eta=f(z).
\end{align*}
For the general situation, let $z_0\in \Omega$ and $\chi\in C_0^\infty(\mathbb{C})$, $0\leq \chi\leq 1$, $\chi=1$ in some neighborhood $V$ of $z_0$ and $\supp \chi\subset \Omega$. Thus,
\begin{align*}
            u(z)&=\frac{1}{\pi}\int_{\Omega}  f(\zeta)\frac{\overline{z-\zeta}}{z-\zeta} d\xi d\eta\notag\\
            &=\frac{1}{\pi}\int_{\Omega} \chi f(\zeta)\frac{\overline{z-\zeta}}{z-\zeta} d\xi d\eta+\frac{1}{\pi}\int_{\Omega}  (1-\chi(\zeta))f(\zeta)\frac{\overline{z-\zeta}}{z-\zeta} d\xi d\eta\notag\\
            &=:u_1(z)+u_2(z).
\end{align*}
Since $\partial^2_{\overline z} u_2=0$ in $V$, from the previous argument for $\Omega=\mathbb{C}$, we have
$$
\partial^2_{\overline z} u(z)=\partial^2_{\overline z} u_1(z)+\partial^2_{\overline z} u_2(z)=\chi(z)f(z)=f(z)
$$
for $z\in V$.

\end{proof}

\begin{lemma}\label{partial}
 Let $\Omega$ be an open bounded domain with smooth boundary $\partial\Omega$.  Suppose that $f,g\in C^2(\Omega)\cap C^1(\overline\Omega)$. Suppose that the compatibility condition
  \begin{equation}\label{comp}
   \partial^2_{\overline z} f=\partial^2_z g\ \ \ \hbox{in $\Omega$}
  \end{equation}
  is satisfied. Then there exists a function $w\in C^2(\Omega)$ satisfies
  \begin{align}
   \left\{
    \begin{array}{cl}
     \partial^2_z w=f\ \ \ \hbox{in $\Omega$},\\
     \partial^2_{\overline z}  w=g\ \ \ \hbox{in $\Omega$}.
     \end{array}
     \right.
  \end{align}
\end{lemma}

\begin{proof}
Let us make an ansatz
\begin{align*}
     w(z)&=\frac{1}{\pi}\int_\Omega g(\zeta) \frac{\overline{z-\zeta}}{z-\zeta} d\xi d\eta
     +\frac{1}{\pi}\int_\Omega f(\zeta) \frac{z-\zeta}{\overline{z-\zeta}} d\xi d\eta\\
          &\quad-\frac{1}{\pi^2}\int_\Omega \left(\int_\Omega \partial^2_{\overline\lambda} f(\lambda) \frac{\overline{\zeta-\lambda}}{\zeta-\lambda} dsdt\right)\frac{z-\zeta}{\overline{z-\zeta}} d\xi d\eta
     +\phi_1(z)+\phi_2(z),
\end{align*}
%$$
 %   \phi_1(z)=\frac{1}{\pi}\int_\Omega \overline\partial^2_\zeta f(\zeta)|z-\zeta|^2\log|z-\zeta|^2d\zeta
 %   -\frac{1}{\pi}\int_\Omega \overline\partial_\zeta f(\zeta)(z-\zeta)d\zeta
 %   -\frac{1}{\pi}\int_\Omega f(\zeta)\frac{z-\zeta}{\overline{z-\zeta}}d\zeta
%$$
%and
%$$
%    \phi_2(z)=\frac{1}{\pi}\int_\Omega \partial^2_\zeta g(\zeta)|z-\zeta|^2\log|z-\zeta|^2d\zeta
%    -\frac{1}{\pi}\int_\Omega \partial_\zeta g(\zeta)(\overline{z-\zeta})d\zeta
 %   -\frac{1}{\pi}\int_\Omega g(\zeta)\frac{\overline{z-\zeta}}{z-\zeta}d\zeta.
%$$
where
\begin{align*}
      \phi_1(z)&=-\frac{1}{(2\pi i)^2}\int_{\partial\Omega}\left(\int_{\partial\Omega}\partial_{\overline\lambda} f(\lambda)\frac{\overline{\zeta-\lambda}}{\zeta-\lambda}d\lambda\right)(z-\zeta)\log(\overline{z-\zeta})d\zeta\\
      &\quad -\frac{1}{(2\pi i)^2}\int_{\partial\Omega}\left(\int_{\partial\Omega} f(\lambda)\frac{1}{\zeta-\lambda}d\lambda\right) (z-\zeta)\log(\overline{z-\zeta})d\zeta\\
      &\quad-\frac{1}{(2\pi i)^2}\int_{\partial\Omega}\left(\int_{\partial\Omega} \partial_{\overline\lambda}f(\lambda)\frac{1}{\zeta-\lambda}d\lambda\right) |z-\zeta|^2\log(\overline{z-\zeta})d\zeta
\end{align*}
and
$$
      \phi_2(z)=-\frac{1}{2\pi i}\int_{\partial\Omega}\partial_\zeta g(\zeta)|z-\zeta|^2\log(z-\zeta)d\overline\zeta
      -\frac{1}{2\pi i}\int_{\partial\Omega} g(\zeta)(\overline{z-\zeta})\log(z-\zeta)d\overline\zeta.
$$
Here we take the principal value for the $\log$. Since $z-\zeta$
does not vanish for all $z\in\Omega$ and $\zeta\in\partial\Omega$,
$h(z,\zeta)=\log(z-\zeta)$ is well-defined on $\Omega\times D$ where
$D=\{\zeta\in \partial\Omega, 0<\arg(z-\zeta)<2\pi\}$. Moreover, for
fixed $\zeta\in \partial\Omega$, the function $h(z,\zeta)$ is
holomorphic in $\Omega$. We can interchange the differentiation and
the integral sign see Chapter 8 in \cite{L} and get
$$
    \partial^2_z\phi_1(z)=0,\ \ \ \partial^2_{\overline{z}}\phi_2(z)=0\ \ \hbox{in $\Omega$}.
$$
On the other hand, we can compute that
\begin{align*}
\partial^2_{\overline{z}}\phi_1(z)&=\frac{1}{(2\pi i)^2}\int_{\partial\Omega} \left(\int_{\partial\Omega} \partial_{\overline\lambda}f(\lambda)\frac{\overline{\zeta-\lambda}}{\zeta-\lambda}d\lambda\right) \partial_{\overline \zeta}\left( \frac{z-\zeta}{\overline{z-\zeta}}\right)d\zeta\\
&\quad +\frac{1}{(2\pi i)^2}\int_{\partial\Omega}\left( \int_{\partial\Omega} f(\lambda)\frac{1}{\zeta-\lambda} d\lambda \right)\partial_{\overline \zeta}\left( \frac{z-\zeta}{\overline{z-\zeta}}\right)d\zeta\\
&\quad -\frac{1}{(2\pi i)^2}\int_{\partial\Omega}\left( \int_{\partial\Omega} \partial_{\overline\lambda}f(\lambda)\frac{1}{\zeta-\lambda} d\lambda \right) \frac{z-\zeta}{\overline{z-\zeta}}d\zeta,
\end{align*}
and
$$
\partial^2_z \phi_2(z)=\frac{1}{2\pi i}\int_{\partial\Omega} g(\zeta)\partial_\zeta\left(\frac{\overline{z-\zeta}}{z-\zeta}\right)d\overline\zeta
-\frac{1}{2\pi i}\int_{\partial\Omega} \partial_\zeta g(\zeta)\frac{\overline{z-\zeta}}{z-\zeta}d\overline\zeta.
$$
Using the compatibility condition \eqref{comp},  Lemma~\ref{lem3.1}, and Gauss's formula (\ref{Gauss}) twice, we can see that
\begin{align*}
&\partial^2_z\left( \frac{1}{\pi^2}\int_\Omega \left(\int_\Omega \partial^2_{\overline\lambda} f(\lambda) \frac{\overline{\zeta-\lambda}}{\zeta-\lambda} dsdt\right)\frac{z-\zeta}{\overline{z-\zeta}} d\xi d\eta\right)\\
&=\partial^2_z\left( \frac{1}{\pi^2}\int_\Omega \left(\int_\Omega \partial^2_{\lambda} g(\lambda) \frac{\overline{\zeta-\lambda}}{\zeta-\lambda} dsdt\right)\frac{z-\zeta}{\overline{z-\zeta}} d\xi d\eta\right)=\frac{1}{\pi}\int_{\Omega}\partial^2_{\zeta} g(\zeta) \frac{\overline{z-\zeta}}{z-\zeta} d\xi d\eta\\
&= \frac{1}{\pi }\int_{\Omega} g(\zeta)\partial^2_\zeta \left(\frac{\overline{z-\zeta}}{z-\zeta}\right)d\xi d\eta
+\frac{1}{2\pi i}\int_{\partial\Omega} g(\zeta)\partial_\zeta\left(\frac{\overline{z-\zeta}}{z-\zeta}\right)d\overline\zeta\\
&\quad -\frac{1}{2\pi i}\int_{\partial\Omega} \partial_\zeta g(\zeta)\frac{\overline{z-\zeta}}{z-\zeta}d\overline\zeta.
\end{align*}
By the above relation and the ansatz, we then deduce
\begin{align*}
\partial^2_z w(z)
&=\frac{1}{\pi}\int_\Omega g(\zeta)\partial^2_z \left(\frac{\overline{z-\zeta}}{z-\zeta}\right)d\xi d\eta+f(z)\\
&\quad-\partial^2_z\left( \frac{1}{\pi^2}\int_\Omega \left(\int_\Omega \partial^2_\lambda g(\lambda) \frac{\overline{\zeta-\lambda}}{\zeta-\lambda} dsdt\right)\frac{z-\zeta}{\overline{z-\zeta}} d\xi d\eta\right)
 +\partial^2_z \phi_1(z)+\partial^2_z \phi_2(z)\\
&=f(z).
\end{align*}

On the other hand, from (\ref{hpomp}), we have that
\begin{align*}
\int_\Omega \partial^2_{\overline\lambda} f(\lambda) \frac{\overline{\zeta-\lambda}}{\zeta-\lambda} dsdt
   =f(\zeta)\pi+\frac{1}{2i} \int_{\partial\Omega} f(\lambda)\left( \frac{1}{\zeta-\lambda} \right)d\lambda
    +\frac{1}{2i} \int_{\partial\Omega}\partial_{\overline\lambda} f(\lambda)\frac{\overline{\zeta-\lambda}}{\zeta-\lambda} d\lambda,
\end{align*}
which implies that
\begin{align*}
h(z)&:=\partial^2_{\overline z}\left( \frac{1}{\pi^2}\int_\Omega \left(\int_\Omega \partial^2_{\overline\lambda} f(\lambda) \frac{\overline{\zeta-\lambda}}{\zeta-\lambda} dsdt\right)\frac{z-\zeta}{\overline{z-\zeta}} d\xi d\eta\right)\\
&= \frac{1}{\pi }\int_{\Omega} f(\zeta)\partial^2_{\overline z}\left( \frac{z-\zeta}{\overline{z-\zeta}}\right)d\xi d\eta\\
&\quad+\frac{1}{2\pi^2 i}\int_\Omega\left(\int_{\partial\Omega} f(\lambda)\left(\frac{1}{\zeta-\lambda}\right)d\lambda+\int_{\partial\Omega} \partial_{\overline\lambda}f(\lambda)\frac{\overline{\zeta-\lambda}}{\zeta-\lambda}d\lambda\right) \partial^2_{\overline z} \left(\frac{z-\zeta}{\overline{z-\zeta}}\right)d\xi d\eta.
\end{align*}
Applying (\ref{Gauss}) twice yields
\begin{align*}
h(z)&=\frac{1}{\pi }\int_{\Omega} f(\zeta)\partial^2_{\overline z}\left( \frac{z-\zeta}{\overline{z-\zeta}}\right)d\xi d\eta\\
&\quad+\frac{1}{(2\pi i)^2}\int_{\partial\Omega}\left( \int_{\partial\Omega} f(\lambda)\left( \frac{1}{\zeta-\lambda} \right)d\lambda\right) \partial_{\overline \zeta}\left( \frac{z-\zeta}{\overline{z-\zeta}}\right)d\zeta\\
&\quad +\frac{1}{(2\pi i)^2}\int_{\partial\Omega}\left( \int_{\partial\Omega}\partial_{\overline\lambda} f(\lambda)\frac{\overline{\zeta-\lambda}}{\zeta-\lambda} d\lambda \right)\partial_{\overline \zeta}\left( \frac{z-\zeta}{\overline{z-\zeta}}\right)d\zeta\\
&\quad-\frac{1}{(2\pi i)^2}\int_{\partial\Omega}\left( \int_{\partial\Omega} \partial_{\overline\lambda}f(\lambda)\frac{1}{\zeta-\lambda} d\lambda \right) \frac{z-\zeta}{\overline{z-\zeta}}d\zeta.
\end{align*}
In view of Lemma~\ref{lem3.1} and $h$, we conclude that
\begin{align*}
\partial^2_{\overline z} w(z)
&=g(z)+\int_\Omega f(\zeta)\partial^2_{\overline z} \left(\frac{z-\zeta}{\overline{z-\zeta}}\right)d\xi d\eta\\
&\quad-\partial^2_{\overline z}\left( \frac{1}{\pi^2}\int_\Omega \left(\int_\Omega \partial^2_{\overline\eta} f(\eta) \frac{\overline{\zeta-\eta}}{\zeta-\eta} dsdt\right)\frac{z-\zeta}{\overline{z-\zeta}} d\xi d\eta\right)
 +\partial^2_{\overline z} \phi_1(z)+\partial^2_{\overline z} \phi_2(z)\\
&=g(z).
\end{align*}

\end{proof}

Note that the above lemma also holds when $f,g\in H^2(\Omega)$ since
we can approximate a $H^2$ function by a sequence in $
C^\infty(\overline\Omega)$ in the $H^2(\Omega)$ space.

%Define the partial complex differentiation
%$\partial_z$ and $\partial_{\overline z}$ by
%$$
%     \partial_z=\frac{1}{2}\left(\frac{\partial}{\partial x}-i\frac{\partial}{\partial y}\right),
%\ \ \ \partial_{\overline z}=\frac{1}{2}\left(\frac{\partial}{\partial x}+i\frac{\partial}{\partial y}\right)
%$$
%in the $z=x+iy$ complex coordinate on $\Omega$.

\subsection{$\partial_{\overline z}$ system}

Let $A$ and $B$ be two $2\times 2$ matrices. We define $A\cdot B=\trace(AB^T).$
We write equation (\ref{divdiv}) in nondivergence form
\begin{align*}
0&=\diver\diver\left(\frac{1}{2\mu}\left(\nabla^2\phi-\frac{\Delta\phi}{2} I_2\right)\right)\\
 &=\frac{1}{4\mu}\Delta^2\phi+\frac{1}{2}\nabla\left(\frac{1}{\mu}\right)\cdot\nabla(\Delta\phi)+\frac{1}{2}\nabla^2\left(\frac{1}{\mu}\right)
 \cdot\left(\nabla^2\phi-\frac{\Delta\phi}{2}I_2\right).
\end{align*}
Since $\mu>0$, the equation above is equivalent to
\begin{align}
\Delta^2\phi+2\mu\nabla\left(\frac{1}{\mu}\right)\cdot\nabla\left(\Delta\phi\right)+2\mu\nabla^2\left(\frac{1}{\mu}\right)
 \cdot\left(\nabla^2\phi-\frac{\Delta\phi}{2}I_2\right)=0,
\end{align}
which implies that
\begin{equation}\label{1403}
     \partial^2_{\overline z}\partial^2_{z}\phi+ \alpha \partial^2_{ z} \partial_{\overline z} \phi +\beta \partial^2_{ z}\phi +\overline\alpha \partial_{z} \partial^2_{\overline z} \phi
      +\overline\beta \partial^2_{\overline z}\phi=0,
\end{equation}
where
\begin{align}\label{alphabeta}
    \alpha=\mu\partial_{\overline z}\left(\frac{1}{\mu}\right),\ \
    \beta=\frac{\mu}{2}\partial_{\overline z}^2\left(\frac{1}{\mu}\right).
\end{align}

With equation \eqref{1403} in mind, we define a first order system $D+V$ acting on functions with values in $\mathbb{C}^4$ as follows
\begin{align}\label{sys}
D+V=\left(
  \begin{array}{ccccc}
    \partial_{\overline z} & 0  & 0 & 0 \\
    0 & \partial_{\overline z}  & 0 & 0 \\
    0 & 0  & \partial_z & 0 \\
    0 & 0  & 0 & \partial_z \\
  \end{array}
\right)+
\left(
  \begin{array}{ccccc}
    \alpha  & \beta & \overline\alpha & \overline\beta \\
    -1 & 0  & 0 & 0 \\
    \alpha  & \beta & \overline\alpha & \overline\beta \\
    0 & 0 & -1 & 0 \\
  \end{array}
\right).
\end{align}
The corresponding Cauchy data of $D+V$ is
$$
     C_{D+V}=\left\{ U|_{\partial\Omega}: U\in H^1(\Omega,\mathbb{C}^4),
      U\ \mbox{is a solution of (D+V)U=0}\right\}.
$$
The next key step is to show that the Cauchy data $C_{\mu}^{\ast}$ for \eqref{divdiv} determine $C_{D+V}$. To do so, we begin the following lemma saying that $C_{\mu}^{\ast}$ determines all derivatives of the solution on the boundary up to third order under suitable assumption.
%In the following result, we show that the boundary data of $\partial^\alpha\phi,\ |\alpha|\leq 3$ is determined by the Cauchy data $C^*_{\mu}$.
\begin{lemma}\label{third}
Assume that $\partial^\kappa\mu_1|_{\partial\Omega}=\partial^\kappa\mu_2|_{\partial\Omega}$ for all $|\kappa|\leq 1$.
If $C_{\mu_1}^{\ast}=C_{\mu_2}^{\ast}$, i.e.,
$$
    \{\phi_1,\phi_{1,n},M_{1,n},(M_t)_{1,t}\}
    =\{\phi_2,\phi_{2,n},M_{2,n},(M_t)_{2,t}\},
$$
where $\phi_j$ is the solution to the equation $P_{\mu_j}(\phi_j)=0$, $j=1,2$, then
$$
\partial^\kappa \phi_1=\partial^\kappa \phi_2\ \ \hbox{on $\partial\Omega$ for $|\kappa|\leq 3$}.
$$
\end{lemma}

\begin{proof}

   % Assume for the moment that $\phi_j\in C^3(\overline\Omega)$. By a density argument, the general case holds.
   % By hypothesis
    %$
    %     C^*_{\mu_1}=C^*_{\mu_2},
    %$
    %if $P_{\mu_1}(\phi_1)=0$, there exists
   % $\phi_2$ such that $P_{\mu_2}(\phi_2)=0$ and
    %$$
   % \{\phi_1,\phi_{1,n},M_{1,n},(M_t)_{1,t}\}
   % =\{\phi_2,\phi_{2,n},M_{2,n},(M_t)_{2,t}\}.
   % $$

    The equalities $\phi_1=\phi_2$ and $\phi_{1,n}=\phi_{2,n}$ gives
    $\nabla \phi_1=\nabla \phi_2$ on $\partial\Omega$, i.e.,
    \begin{equation*}
          \phi_{1,1}=\phi_{2,1},\ \ \phi_{1,2}=\phi_{2,2} \ \ \hbox{on $\partial\Omega$}
    \end{equation*}
    and thus
    \begin{align}\label{second}
          \nabla\phi_{1,k}\cdot \textbf{t}=\nabla\phi_{2,k}\cdot \textbf{t},\ \ \ \ k=1,2\ \ \hbox{on $\partial\Omega$}.
    \end{align}
    Moreover, since $M_{1,n}=M_{2,n}$, by the definition of $M_n$ and the hypothesis
    $\mu_1|_{\partial\Omega}=\mu_2|_{\partial\Omega}$, we obtain
    \begin{align}\label{Mn}
    (n_1^2-n_2^2)(\phi_{1,11}-\phi_{2,11}) -(n_1^2-n_2^2)(\phi_{1,22}
    -\phi_{2,22})+4n_1n_2(\phi_{1,12}-\phi_{2,12})=0.
    \end{align}
    From (\ref{second}) and (\ref{Mn}), we have
    \begin{align}
        AU:=\left(
          \begin{array}{ccc}
            -n_2 & n_1 & 0 \\
            0 & -n_2 & n_1 \\
            n^2_1-n^2_2 & 4n_1n_2 &  n^2_2-n^2_1\\
          \end{array}
        \right)
                \left(
                 \begin{array}{c}
                   \phi_{1,11}-\phi_{2,11} \\
                   \phi_{1,12}-\phi_{2,12} \\
                   \phi_{1,22}-\phi_{2,22} \\
                 \end{array}
               \right)=0\ \ \ \hbox{on $\partial\Omega$}.
    \end{align}
    Since the matrix $A$ is invertible, we get that
    $\phi_{1,ij}=\phi_{2,ij}$ on $\partial\Omega$ for $1\leq i,j\leq 2$.

    With $\phi_{1,ij}=\phi_{2,ij}$ on $\partial\Omega$, we can deduce
$$
\nabla\phi_{1,ij}\cdot \textbf{t}=\nabla\phi_{2,ij}\cdot \textbf{t},
$$
that is,
\begin{align}\label{31}
    -n_2\phi_{1,1ij}+n_1\phi_{1,2ij}=-n_2\phi_{2,1ij}+n_1\phi_{2,2ij}.
\end{align}
Using the condition $(M_t)_{1,t}=(M_t)_{2,t}$ and
$\phi_{1,ij}=\phi_{2,ij}$ on $\partial\Omega$ for $1\leq i,j\leq 2$,
it follows that
\begin{align}\label{32}
     \diver\left(\frac{1}{2\mu_1} (\nabla^2\phi_1-\frac{\Delta \phi_1}{2}I_2)\right)\cdot\textbf{n}
     = \diver\left(\frac{1}{2\mu_2} (\nabla^2\phi_2-\frac{\Delta \phi_2}{2}I_2)\right)\cdot\textbf{n}.
\end{align}
Putting (\ref{31}), (\ref{32}) together and using the boundary assumption of $\mu$, we obtain that
$$
\left(
  \begin{array}{cccc}
    -n_2 & n_1 & 0 & 0 \\
    0 & -n_2 & n_1 & 0 \\
    0 & 0 & -n_2 & n_1 \\
    n_1 & n_2 & n_1 & n_2 \\
  \end{array}
\right)
\left(
  \begin{array}{c}
    \phi_{1,111}-\phi_{2,111} \\
    \phi_{1,112}-\phi_{2,112} \\
    \phi_{1,122}-\phi_{2,122} \\
    \phi_{1,222}-\phi_{2,222} \\
  \end{array}
\right)=0.
$$
Since the matrix above is invertible, we deduce that $\phi_{1,ijk}=\phi_{2,ijk}$ for $1\leq i,j,k\leq 2$.

\end{proof}

We are now ready to prove the crucial step.
%With the above lemma, we are ready to show the equivalence of $C^*_{\mu}$ and $C_{D+V}$.
\begin{lemma}\label{DV}
     Assume that $\mu\in C^3(\overline\Omega)$. Suppose that $\partial^\kappa\mu_1|_{\partial\Omega}=\partial^\kappa\mu_2|_{\partial\Omega},\ \forall\ |\kappa|\leq 1$.
     The Cauchy data $C^*_{\mu}$ of $P_\mu$ determines the Cauchy data $C_{D+V}$ of
     $D+V$.
\end{lemma}
\begin{proof}
Assume that $C^*_{\mu_1}=C^*_{\mu_2}$ with two parameters $\mu_1$ and $\mu_2$. Let $U_1=(u_1,u_2,u_3,u_4)^T$ be a solution of $(D+V_1)U_1=0$, then
$$
(D+V_1)U_1=\left(\left(
  \begin{array}{ccccc}
    \partial_{\overline z} & 0  & 0 & 0 \\
    0 & \partial_{\overline z}  & 0 & 0 \\
    0 & 0  & \partial_z & 0 \\
    0 & 0  & 0 & \partial_z \\
  \end{array}
\right)+
\left(
  \begin{array}{ccccc}
    \alpha_1  & \beta_1 & \overline\alpha_1 & \overline\beta_1 \\
    -1 & 0  & 0 & 0 \\
    \alpha_1 & \beta_1 & \overline\alpha_1 & \overline\beta_1 \\
    0 & 0 & -1 & 0 \\
  \end{array}
\right)\right)
\left(
  \begin{array}{c}
    u_1 \\
    u_2 \\
    u_3 \\
    u_4 \\
  \end{array}
\right)
=0,
$$
where $\alpha_j,\beta_j$ are defined in \eqref{alphabeta} with respect to $\mu_j$, $j=1,2$, respectively. The 2nd and 4th equations of the system $(D+V_1)U_1=0$ gives
\begin{align}\label{24}
 \partial_{\overline z} u_2=u_1,\ \ \ \partial_z u_4=u_3.
\end{align}
Likewise, the 1st and 3rd equations of $(D+V_1))U_1=0$ implies
\begin{align}\label{13}
  \partial_{\overline z} u_1=\partial_z u_3
\end{align}
It immediately follows from (\ref{24}) and (\ref{13}) that
\begin{align}\label{comb}
  \partial_{\overline z}^2 u_2=\partial_z^2 u_4.
\end{align}
In view of \eqref{comb} and Lemma~\ref{partial}, there exists a function $\Phi_1$ satisfying
\begin{align}\label{not1}
     \partial^2_z \Phi_1=u_2,\ \ \partial_{\overline z}^2 \Phi_1=u_4.
\end{align}
Substituting \eqref{not1} into (\ref{24}) gives
\begin{align}\label{not2}
     u_1=\partial_{\overline z} u_2=\partial^2_{ z} \partial_{\overline z} \Phi_1,\ \ u_3=\partial_z u_4=\partial_{ z} \partial^2_{\overline z}\Phi_1.
\end{align}

The 1st equation of $(D+V_1)U_1=0$, i.e.,
$$
    \partial_{\overline z} u_1+\alpha u_1+\beta u_2+\overline\alpha u_3+\overline\beta u_4=0
$$
with $u_1,\cdots,u_4$ replaced by \eqref{not1} and \eqref{not2} above, is equivalent to
$$
   P_{\mu_1}(\Phi_1)=0\quad\text{in}\quad\Omega
$$
(cf. \eqref{1403}). %More precisely, the real part of $\Phi_1$, $\Re{\Phi_1}$, and the
%imaginary part of $\Phi_1$, $\Im{\Phi_1}$, satisfy the equation
%$P_{\mu_1}(\phi)=0$.
Similarly, for $V_2$ and $U_2$ satisfying $(D+V_2)U_2=0$ in $\Omega$ associated with $\mu_2$, we obtain a $\Phi_2$ solving $P_{\mu_2}(\Phi_2)=0$ in $\Omega$ where the components of $U_2$ and $\Phi_2$ satisfy corresponding equations like \eqref{not1}, \eqref{not2}.  The assumption $C^*_{\mu_1}=C^*_{\mu_2}$ implies
$$
\{\Phi_1,\Phi_{1,n},M_{1,n},(M_t)_{1,t}\}=\{\Phi_2,\Phi_{2,n},M_{2,n},(M_t)_{2,t}\}
$$
and Lemma~\ref{third} gives
$$
\partial^\kappa(\Phi_1-\Phi_2)|_{\partial\Omega}=0 \ \ \ \hbox{for
$|\kappa|\leq 3$}.
$$
Since $U_2=(\partial^2_{ z} \partial_{\overline z}\Phi_2,\partial_z^2\Phi_2,\partial_{ z} \partial^2_{\overline z}\Phi_2,\partial_{\overline z}^2\Phi_2)^T$, we obtain $(U_1-U_2)|_{\partial\Omega}=0$ and thus $C_{D+V_1}=C_{D+V_2}$.

\end{proof}

\subsection{Proof of the uniqueness result}

We denote
\begin{align*}
A=\left(
    \begin{array}{cc}
      \alpha & \beta \\
      -1 & 0 \\
    \end{array}
  \right)
  \ \ \hbox{and}\ \
  Q=\left(
    \begin{array}{cc}
      \alpha & \beta \\
      0 & 0 \\
    \end{array}
  \right),
\end{align*}
then the system $D+V$ can be represented as
$$
    D+V=
\left(
    \begin{array}{cc}
      \partial_{\overline z}I_{2} & 0 \\
      0 & \partial_z I_2\\
    \end{array}
  \right)
  +
  \left(
    \begin{array}{cc}
      A & \overline Q \\
      Q & \overline A \\
    \end{array}
  \right).
$$
In the following lemma, we show that $\mu$ is uniquely determined by the Cauchy data $C_{D+V}$.
\begin{lemma}\label{muid}
    Let $(\alpha_j,\beta_j), j=1,2$ be in $C^1(\overline\Omega)$.
    Assume that $\partial^\kappa\mu_1|_{\partial\Omega}=\partial^\kappa\mu_2|_{\partial\Omega}$ for all $|\kappa|\leq 1$. If $C_{D+V_1}=C_{D+V_2}$, then
     $
    \mu_1=\mu_2
     $
     in $\Omega$.
\end{lemma}
\begin{proof}
Using that  $C_{D+V_1}=C_{D+V_2}$, we apply Theorem 4.1 in \cite{AGTU} to obtain that there exist invertible matrices
$F_j\in C^1(\Omega,\mathbb{C}^2\oplus\mathbb{C}^2)$ such that
$F_1=F_2$ on $\partial\Omega$. Moreover,
\begin{equation}\label{faq}
     \partial_{\overline z} F_j=F_jA_j\ \ \hbox{and}\ \ Q_1=\overline FQ_2F^{-1},
\end{equation}
where $F:=F^{-1}_1F_2$ is an invertible matrix.

Let us denote the two rows of the matrix $F_j^{-1}$ by $a_j$ and
$b_j$, then the first relation of \eqref{faq} implies $\partial_{\overline z} F_j^{-1}=-A_jF^{-1}_j$ and hence
\begin{align}\label{Fj}
F_j^{-1}= \left(
  \begin{array}{c}
    \partial_{\overline z} b_j\\
    b_j \\
  \end{array}
\right)
\end{align}
with the help of the form of $A_j$. We now write
$$
F^{-1}=
\left(
  \begin{array}{cc}
    h & v \\
    m & r \\
  \end{array}
\right).
$$
Using the condition $Q_1=\overline FQ_2F^{-1}$, we have that
\[
     \overline m\alpha_1=\overline m\beta_1=0
\]
 and
 \begin{equation}\label{a2b2}
     \overline h\alpha_1=h\alpha_2+m\beta_2,\ \ \overline h\beta_1=v\alpha_2+r\beta_2.
\end{equation}
Then $m=0$ in $\Omega'$, where $\Omega'=\{x\in\Omega: \alpha_1(x)\neq 0\ \hbox{or}\ \beta_1(x)\neq 0\}$.
Note that if $x$ is in the
complement of $\Omega'$, then $(\alpha_2(x),\beta_2(x))$ must be zero
by \eqref{a2b2} since $F$ is invertible. Thus $\alpha_1=\alpha_2=0$ in the
complement of $\Omega'$. If $\Omega'$ is empty, then $\alpha_1=0=\alpha_2$ in $\Omega$. By the boundary condition $\partial^\kappa\mu_1=\partial^\kappa\mu_2$ for $|\kappa|\leq 1$, we conclude that $\mu_1=\mu_2$. Actually, in this case, we obtain that $\mu_1=\mu_2=\text{constant}$.

Now we suppose that $\Omega'$ is a nonempty open set. Since $m=0$ in $\Omega'$, $F^{-1}$ can be rewritten as
$$
F^{-1}=
\left(
  \begin{array}{cc}
    h & v \\
    0 & r \\
  \end{array}
\right)
$$
in $\Omega'$. Using $F=F^{-1}_1F_2$ and (\ref{Fj}), we can deduce
that
$$
    h\partial_{\overline z} b_1+v b_1=\partial_{\overline z}b_2,\ \ rb_1=b_2,\ \ \hbox{in $\Omega'$},
$$
which implies
\begin{align}\label{Finv}
F^{-1}=
\left(
  \begin{array}{cc}
    r & \partial_{\overline z} r \\
    0 & r \\
  \end{array}
\right)\ \ \hbox{in $\Omega'$}.
\end{align}
In deriving \eqref{Finv}, we used the fact that $\partial_{\overline z}b_1$ and $b_1$ are linearly independent due to the invertibility of $F_1^{-1}$.
%Since $F$ is holomorphic, it follows that $r$ is holomorphic, which leads to $\partial_{\overline z}r=0$ in $\Omega'$.
%Using (\ref{Finv}) and the following identity
%\begin{align*}
%   \frac{\overline\partial\det F}{\det F}=\frac{\overline\partial\det F_2}{\det F_2}
%   -\frac{\overline\partial\det F_1}{\det F_1}=\trace{(A_2-A_1)}=\alpha_2-\alpha_1,
%\end{align*}
%we have
%\begin{align}\label{r}
%     2\partial_{\overline z} r=(\alpha_1-\alpha_2)r\ \ \hbox{in $\Omega'$}.
%\end{align}
Note that since $F$ is invertible, $r$ never vanishes at any point in $\Omega'$.

We observe that
\begin{align*}
     \partial_{\overline z}F^{-1}
&=\partial_{\overline z}(F_2^{-1}F_1)
=\partial_{\overline z}F_2^{-1}F_1
+F_2^{-1}\partial_{\overline z}F_1\\
&=-A_2F^{-1}+F^{-1}A_1,
\end{align*}
then it follows that
\begin{align}\label{r1}
    2 \partial_{\overline z} r=(\alpha_1-\alpha_2)r
\end{align}
and
\begin{align}\label{r2}
    \partial_{\overline z}^2 r=(\beta_1-\beta_2)r-\alpha_2\partial_{\overline z} r.
\end{align}
From (\ref{r1}), we have
\begin{align}\label{r3}
     2 \partial^2_{\overline z} r=r\partial_{\overline z}(\alpha_1-\alpha_2)+(\alpha_1-\alpha_2)\partial_{\overline z}r.
\end{align}
Substituting $(\partial_{\overline z}r=(\alpha_1-\alpha_2)r/2)$ into (\ref{r2}) and (\ref{r3}) gives
\begin{align*}
     2 \partial^2_{\overline z} r=\left( 2\beta_1-2\beta_2-\alpha_2(\alpha_1-\alpha_2)   \right)r
=\left( \partial_{\overline z}(\alpha_1-\alpha_2)+(\alpha_1-\alpha_2)^2/2   \right)r,
\end{align*}
which implies
\begin{align}\label{r4}
 2\beta_1-2\beta_2-\alpha_2(\alpha_1-\alpha_2)  =\partial_{\overline z}(\alpha_1-\alpha_2)+(\alpha_1-\alpha_2)^2/2 .
\end{align}
Note that $r$ does not vanish in $\Omega'$.
By direct computation and the definition of $\alpha_j$ and $\beta_j$ in (\ref{alphabeta}), it follows that
\begin{align}\label{rr}
    2\beta_j=\partial_{\overline z} \alpha_j+\alpha_j^2.
\end{align}
Then we obtain
\begin{align}\label{alpha}
   \alpha_1^2= \alpha_2^2\;\; \text{in}\;\; \Omega'
\end{align}
by substituting (\ref{rr}) into (\ref{r4}).
Combining \eqref{alpha} and the previously derived fact
\begin{align*}
\alpha_1=\alpha_2=0\;\;\text{in}\;\;\Omega\setminus\Omega',
\end{align*}
we have that
$$
    \alpha^2_1=\alpha^2_2\;\;\text{in}\;\; \Omega,
$$
which is equivalent to
$$
     ( \nabla \log\mu_1)^2=( \nabla \log\mu_2)^2\;\;\text{in}\;\;\Omega.
$$
Since $\mu_1|_{\partial\Omega}=\mu_2|_{\partial\Omega}$ and by the continuity of $\mu_j$ and $\nabla\mu_j$, $j=1,2$, we obtain
$$
      \nabla \log\mu_1=\nabla\log\mu_2\;\;\text{in}\;\;\Omega.
$$
%Since $\partial^\delta\mu_1|_{\partial\Omega}=\partial^\delta\mu_2|_{\partial\Omega}$ for $|\delta|\leq 1$, we obtain $\alpha_1=\alpha_2$ on $\partial\Omega$.
Using the boundary condition $\mu_1|_{\partial\Omega}=\mu_2|_{\partial\Omega}$ again, we finally conclude that
$\mu_1=\mu_2$ in $\Omega$.

\end{proof}

%\emph{Proof of the main theorem \ref{main}.}
%\begin{proof}
%Since $C^*_{\mu_1}=C^*_{\mu_2}$ implies $C_{D+V_1}=C_{D+V_2}$, by Lemma \ref{muid}, we conclude that $\mu_1=\mu_2$ in
%$\Omega$.
%\end{proof}

%\begin{remark}
%We use $Q_1=\overline FQ_2F^{-1}$ again, then
%\begin{align}\label{boundary}
%   \alpha_1\overline r=\alpha_2 r,\ \ \beta_1 \overline r=\alpha_2\partial_{\overline z} r+\beta_2 r.
%\end{align}
%We observe the behavior on $\partial\Omega$. Since $F=Id$ on
%$\partial\Omega$ and (\ref{boundary}), we obtain
%$$
%r=1,\ \ \partial_{\overline z} r=0\ \ \hbox{on
%$\partial\Omega$}.
%$$
%Using (\ref{boundary}), we deduce
%$$
%    \partial_{\overline z}^{\alpha}\alpha_1|_{\partial\Omega} =\partial_{\overline z}^{\alpha}\alpha_2|_{\partial\Omega},\ \ \ |\alpha|\leq 1.
%$$
%\end{remark}

\emph{Proof of theorem \ref{main}.}
From Section \ref{sec2} we have known that the Cauchy data for the Stokes equations and that for the equation $P_\mu(\phi)=0$ are equivalent, that is, $C_{\mu_1}=C_{\mu_2}$ is equivalent to $C^*_{\mu_1}=C^*_{\mu_2}$.
Therefore, Theorem \ref{main} follows from Lemma \ref{DV} and Lemma \ref{muid}.

\section{Global uniqueness for the stationary Navier-Stokes equations}\label{sec4}
In this section we consider the unique determination of the viscosity in an incompressible fluid described by the stationary Navier-Stokes equations. In higher dimensions, this problem has been solved by Li and Wang in \cite{LW} using the linearization technique. Since their methods are independent of spatial dimensions, we could apply their ideas to show the uniqueness result of $\mu$ for the Navier-Stokes equations in the two dimensional case.

Let $u=(u_1,u_2)^T$ be the velocity vector field satisfying the stationary Navier-Stokes equations
\begin{align}\label{navier}
       \left\{
          \begin{array}{cl}
         \diver\sigma(u,p)-(u\cdot \nabla)u=0 & \hbox{in $\Omega$},\\
            \diver u=0  & \hbox{in $ \Omega$},
           \end{array}
        \right.
\end{align}
and the corresponding Cauchy data is denoted by
$$
     \tilde C_{\mu}=\left\{ \left(u, \sigma(u,p)\textbf{n}\right)|_{\partial\Omega}:
      (u,p)\ \mbox{satisfies (\ref{navier})}  \right\}.
$$

Let $u|_{\partial\Omega}=\phi\in H^{3/2}(\partial\Omega)$ satisfy (\ref{compat}). We choose $\phi=\varepsilon\psi$ with $\psi\in H^{3/2}(\partial\Omega)$ and let $(u_\varepsilon, p_\varepsilon)=(\varepsilon v_\varepsilon, \varepsilon q_\varepsilon)$ satisfy (\ref{navier}). The problem (\ref{navier}) is reduced to
\begin{align}\label{navier1}
       \left\{
          \begin{array}{cl}
         \diver\sigma(v_\varepsilon,q_\varepsilon)-\varepsilon(v_\varepsilon\cdot \nabla)v_\varepsilon=0 & \hbox{in $\Omega$},\\
            \diver v_\varepsilon=0  & \hbox{in $ \Omega$},\\
            v_\varepsilon=\psi  & \hbox{on $\partial\Omega$}.
           \end{array}
        \right.
\end{align}
We are looking for a solution of (\ref{navier1}) with the form $v_\varepsilon=v_0+\varepsilon v$ and $q_\varepsilon=q_0+\varepsilon q$, where $(v_0,q_0)$ satisfies the Stokes equations
\begin{align}\label{stokes1}
       \left\{
          \begin{array}{cl}
         \diver\sigma(v_0,q_0)=0 & \hbox{in $\Omega$},\\
            \diver v_0=0  & \hbox{in $ \Omega$},\\
            v_0=\psi & \hbox{on $\partial\Omega$},
           \end{array}
        \right.
\end{align}
and $(v,q)$ satisfies
\begin{align}\label{asym}
       \left\{
          \begin{array}{cl}
         -\diver\sigma(v,q)+\varepsilon (v_0\cdot\nabla)v +\varepsilon(v\cdot\nabla)v_0+\varepsilon^2(v\cdot\nabla)v=f & \hbox{in $\Omega$},\\
            \diver v=0  & \hbox{in $ \Omega$},\\
            v=0 & \hbox{on $\partial\Omega$},
           \end{array}
        \right.
\end{align}
with $f=-(v_0\cdot \nabla)v_0$.

In \cite{LW}, it is shown that for any $\psi\in H^{3/2}(\partial\Omega)$, let $(v_0,q_0)\in H^2(\Omega)\times H^1(\Omega)$ be the unique solution ($q_0$ is unique up to a constant) of the Stokes equations (\ref{stokes1}). There exists a solution $(u_\varepsilon,p_\varepsilon)$ of (\ref{navier}) of the form
$$
      u_\varepsilon=\varepsilon v_0+\varepsilon^2 v,\ \ p_\varepsilon=\varepsilon q_0+\varepsilon^2 q
$$
with the boundary data $u_\varepsilon|_{\partial\Omega}=\varepsilon\psi$ for all $|\varepsilon|\leq \varepsilon_0$, where $\varepsilon_0$ depends on $\|\psi\|_{H^{3/2}(\partial\Omega)}$.
Here $(v,p)$ is a solution of (\ref{asym}) and satisfies the regularity result
$$
    \|v\|_{H^2(\Omega)}+\|q\|_{H^1(\Omega)/\mathbb{R}}\leq C \sum^{16}_{j=2} \|\psi\|^j_{H^{3/2}(\partial\Omega)}
$$
where $C$ is independent of $\varepsilon$ and $\|q\|_{H^1(\Omega)/\mathbb{R}}:=\inf_{c\in\mathbb{R}}\|q+c\|_{H^1(\Omega)}$.
Hence, we have
\begin{align*}
     &\|\varepsilon^{-1} u_\varepsilon-v_0\|_{H^2(\Omega)}=\|\varepsilon v\|_{H^2(\Omega)}\rightarrow 0,\\
     &\|\varepsilon^{-1} p_\varepsilon-q_0\|_{H^1(\Omega)/\mathbb{R}}=\|\varepsilon q\|_{H^1(\Omega)/\mathbb{R}}\rightarrow 0,
\end{align*}
as $\varepsilon\rightarrow 0$, which imply
\begin{align}\label{4:vare1}
     &\|\varepsilon^{-1} u_\varepsilon|_{\partial\Omega}-v_0|_{\partial\Omega}\|_{H^{3/2}(\partial\Omega)}\rightarrow 0,
\end{align}
and
\begin{align}\label{4:vare2}
     &\|\varepsilon^{-1} \sigma(u_\varepsilon,p_\varepsilon)\textbf{n}|_{\partial\Omega}-\sigma(v_0,q_0)\textbf{n}|_{\partial\Omega}\|_{H^{1/2}(\partial\Omega)}\rightarrow 0,
\end{align}
provided
$$
    \int_\Omega p_\varepsilon dx=\int_\Omega q_0 dx=0.
$$
From (\ref{4:vare1}) and (\ref{4:vare2}), we can deduce that the Cauchy data $\tilde C_\mu$ of the Navier-Stokes equations uniquely determines the Cauchy data $C_\mu$ of the Stokes equations. In other words, $\tilde C_{\mu_1}=\tilde C_{\mu_2}$ implies $C_{\mu_1}=C_{\mu_2}$. Therefore, the uniqueness of the viscosity for the Navier-Stokes equations follows from Theorem \ref{main}. We have the following theorem.

\begin{theorem}\label{main2}
     Let $\Omega$ be a simply connected bounded domain in $\mathbb{R}^2$ with smooth boundary. Suppose that $\mu_1$ and $\mu_2$ are two viscosity functions for the Navier-Stokes equations. Assume that $\mu_j\in C^3(\overline\Omega)$ and $\mu_j>0$ with
     \begin{align*}
     \partial^\kappa\mu_1|_{\partial\Omega}=\partial^\kappa\mu_2|_{\partial\Omega}\ \ \hbox{for all $|\kappa|\leq 1$.}
     \end{align*}
     Let $\tilde C_{\mu_1}$ and $\tilde C_{\mu_2}$ be the Cauchy data associated with $\mu_1$ and $\mu_2$, respectively. If $\tilde C_{\mu_1}=\tilde C_{\mu_2}$, then
$
\mu_1=\mu_2
$
in $\Omega$.
\end{theorem}


\begin{thebibliography}{99}
\bibitem{AGTU}
     \newblock P. Albin, C. Guillarmou, L. Tzou and G. Uhlmann, % first name middle initial. and then last name.  Only the first character in the paper title is capitalized.
     \newblock \emph{Inverse boundary problems for systems in two dimensions},
     \newblock Annales Henri Poincar\'e, \textbf{14} (2013), 1551--1571.

%\bibitem{C}
%     \newblock A. Calder\'{o}n, % first name middle initial. and then last name.  Only the first character in the paper title is capitalized.
%     \newblock \emph{On an inverse boundary value problem},
%    \newblock Seminar in Numerical Analysis and its Applications to Continuum Physics (R\'{i}o de Janeiro: Soc.
% Brasileira de Matem\'{a}tica), (1980), 65--73.
\bibitem{B}
\newblock H. Begehr,
\newblock\emph{Representations in polydomains},
\newblock Acta Mathematica Vietnamica, \textbf{27} (2002), 271--282.

\bibitem{CS}
\newblock S.-C. Chen and M.-C. Shaw,
\newblock\emph{Partial Differential Equations in Several Complex Variables},
\newblock American Mathematical Society - International Press, Studies in Advanced Mathematics, Vol. 19, 2001.


\bibitem{F}
     \newblock Y. C. Fung, % first name middle initial. and then last name.  Only the first character in the paper title is capitalized.
     \newblock \emph{Foundations of Solid Mechanics},
     \newblock Prentice-Hall, Inc., Englewood Cliffs, New Jersey, (1965).

\bibitem{HLW}
     \newblock H. Heck, X. Li, and J.-N. Wang, % first name middle initial. and then last name.  Only the first character in the paper title is capitalized.
     \newblock \emph{Identification of viscosity in an incompressible fluid},
     \newblock Indiana University Mathematics Journal, \textbf{56} (2006), 2489--2510.

\bibitem{ik94}
\newblock M. Ikehata,
\newblock\emph{A relationship between two nondestructive testings for the determination of the elasticity tensor field of the elastic thin plate},
\newblock http://math.dept.eng.gunma-u.ac.jp/~ikehata/elasticDecember1994withoutaddress.pdf  (1994)

\bibitem{IY}
     \newblock O. Yu. Imanuvilov and M. Yamamoto, % first name middle initial. and then last name.  Only the first character in the paper title is capitalized.
     \newblock \emph{Global uniqueness in inverse boundary value problems for Navier-Stokes equations and Lam\'e ststem in two dimensions},
     \newblock arXiv:1309.1694, (2013).

\bibitem{is93}
\newblock V. Isakov,
\newblock\emph{On uniqueness in inverse problems for semilinear parabolic equations},
\newblock Arch. Rat. Mech. Anal. \textbf{124} (1993), 1Ð12.


\bibitem{KMW}
     \newblock H. Kang, G. Milton and J.-N. Wang, % first name middle initial. and then last name.  Only the first character in the paper title is capitalized.
     \newblock \emph{Equivalence of inverse problems for $2$D elasticity and for the thin plate with finite meaurements and its applications},
     \newblock arXiv:1203.3833, (2012).

\bibitem{L}
     \newblock S. Lang, % first name middle initial. and then last name.  Only the first character in the paper title is capitalized.
     \newblock \emph{Complex Analysis},
     \newblock Springer-Verlag, (1999).

\bibitem{LW}
     \newblock X. Li and J.-N. Wang, % first name middle initial. and then last name.  Only the first character in the paper title is capitalized.
     \newblock \emph{Determination of viscosity in the stationary Navier-Stokes equations},
     \newblock J. Differential Equations, \textbf{242} (2007), 24-39.

\bibitem{M}
     \newblock G. Milton, % first name middle initial. and then last name.  Only the first character in the paper title is capitalized.
     \newblock \emph{The Theory of Composites},
     \newblock Cambridge Monographs on Applied and Computational Mathematics, Cambridge University Press, (2002).

%\bibitem{U}
%    \newblock G. Uhlmann,
%    \newblock \emph{Electrical impedance tomography and Calder\'on's problem},
%    \newblock  Inverse Problems, 25th Anniversary Volume, \textbf{25} (2009), 123011.

%\bibitem{V}
%     \newblock I. Vekua, % first name middle initial. and then last name.  Only the first character in the paper title is capitalized.
%     \newblock \emph{Generalized Analytic Functions},
 %   \newblock Pergamon Press, Oxford, (1962).

\end{thebibliography}
\end{document}